\documentclass{amsart}
\usepackage{amsmath,amsthm,amssymb,amsfonts,amscd, array, graphicx,float,comment}
\usepackage[all]{xy}
\allowdisplaybreaks

\newtheorem{thm}{Theorem}[section]
 \newtheorem{cor}[thm]{Corollary}
 
 \newtheorem{prop}[thm]{Proposition}
 \theoremstyle{definition} 
 \newtheorem{defn}[thm]{Definition}
\newtheorem{rem}[thm]{Remark}
 \newtheorem{ex}[thm]{Example}
\title[HBT for non-associative power series and Laurent series ring]{THE HILBERT BASIS THEOREM FOR NON-ASSOCIATIVE SKEW POWER SERIES RINGS AND NON-ASSOCIATIVE SKEW LAURENT SERIES RINGS OVER A NON-ASSOCIATIVE COEFFICIENT RING}
\author{Darwin P. Mangubat}
\address[Darwin P. Mangubat, corresponding author]{%
Department of Mathematics and Statistics\\
College of Science and Mathematics\\
Center of Mathematical and Theoretical Physical Sciences-PRISM\\
Mindanao State University-Iligan Institute of Technology\\
9200 Iligan City, Philippines
}
\email{darwin.mangubat@g.msuiit.edu.ph}

\author{Jocelyn P. Vilela}
\address[Jocelyn P. Vilela]{%
Department of Mathematics and Statistics\\
College of Science and Mathematics\\
Center of Mathematical and Theoretical Physical Sciences-PRISM\\
Mindanao State University-Iligan Institute of Technology\\
9200 Iligan City, Philippines
}
\email{jocelyn.vilela@g.msuiit.edu.ph}

\author{Johan Richter}
\address[Johan Richter]{%
Department of Mathematics and Natural Sciences\\
Blekinge Institute of Technology\\
371 79 Karlskrona, Sweden
}
\email{johan.richter@bth.se}

\begin{document}

\begin{abstract}
In this paper, we give conditions on a non-associative ring $R$ and additive map $\sigma: R\longrightarrow R$ so that the non-associative skew power series ring, denoted by $R[[X;\sigma]]$ and the non-associative skew Laurent series ring denoted by $R((X;\sigma))$ are left (resp. right) Noetherian, respectively. Consequently, we establish versions of Hilbert Basis Theorem of $R[[X;\sigma]]$ and $R((X;\sigma))$, respectively, for left (resp. right) Noetherianity .  We also give a counterexample showing that the Hilbert Basis Theorem does not hold in general in the non-associative case. In establishing the Hilbert Basis Theorems for $R[[X;\sigma]]$ and $R((X;\sigma))$, we introduce and utilize the right (resp. left) quasi-associative property of $R$, which is equivalent to $rR$ (resp. $Rc$) being one-sided ideals for all $r\in R$ (resp. $c\in R$).
 \\[+2mm]
 \subjclassname{ 16W30} 
 \end{abstract}

\maketitle

\section{Introduction }

The Hilbert Basis Theorem is a foundational result in the theory of polynomial rings originating from the work of David Hilbert in 1890 \cite{Hilbert}. The original formulation of the theorem dealt with sequences of homogeneous polynomials in the commutative setting. When Noetherianity was first introduced by Emmy Noether in 1921 in her paper \textit{Idealtheorie in Ringbereichen} \cite{Noether}, the Hilbert Basis Theorem was reformulated using Noetherianity. Eventually, the Hilbert Basis Theorem was expanded to the noncommutative case using algebraic structures called \textit{Ore extensions}. 
%Ore extensions were first introduced by {\O}ystein Ore in 1933. 
An Ore extension over any ring $R$, denoted by $R[X;\sigma,\delta]$, is a ring consisting of polynomials over $R$, where $\sigma: R\longrightarrow R$ is an endomorphism and $\delta: R\longrightarrow R$ is a $\sigma$-derivation. $R[X;\sigma,\delta]$ adheres to the multiplication rule $Xr=\sigma(r)X+\delta(r)$ for all $r\in R$ \cite{Back2024}.

In recent studies, Ore extensions have been generalized to the case of non-associative  multiplication and extended to formal power series and formal Laurent series. The study of B\"ack and Richter in 2024 \cite{Back2024} combined these ideas to define new structures called \textit{non-associative skew power series ring $R[[X;\sigma]]$} and\textit{ non-associative skew Laurent series ring $R((X;\sigma))$}. It also established the Hilbert Basis Theorem for both structures, provided that $R$ is an associative ring.

In this paper, we investigate conditions on a non-associative ring $R$ and map $\sigma:R\longrightarrow R$ under which $R[[X;\sigma]]$ and $R((X;\sigma))$ are right and left Noetherian, respectively, and establish a corresponding Hilbert Basis Theorem for right and left Noetherianity. In doing this, we extend the results of B\"ack and Richter \cite{Back2024}. We also give a counterexample showing that the Hilbert Basis Theorem does not hold in general for non-associative skew power series rings. (In fact, our counterexample is for a non-associative formal power series ring.) 

In further detail, this article is outlined as follows: Section \ref{s2} provides preliminaries from the non-associative ring theory, as well as the definitions for $R[[X;\sigma]]$ and $R((X;\sigma))$. Section \ref{s3} establishes the preliminary results of the paper, particularly, the right and left quasi-associative properties of a non-associative ring $R$, and Section 4 presents the Hilbert Basis Theorem for $R[[X;\sigma]]$ (Theorems  \ref{314} and \ref{3110}) and $R((X;\sigma))$ (Corollaries \ref{3119} and \ref{3120left} ) with some concrete examples. Additionally, we show a counterexample (Example \ref{counter}) showing that the Hilbert Basis Theorem fails to hold in some $R[[X;\sigma]]$.

\section{Preliminaries}\label{s2}

Here, we introduce the notion of a non-associative ring, and its related terms and properties, adapting the discussion from the paper of \"Oinert et al. \cite{Nystedt2018} in 2018, unless otherwise specified.
\begin{defn}\label{021}
	A \textit{non-associative} ring $R$ \label{naring} means that $R$ is an additive abelian group, with a defined multiplication satisfying left and right distributivity having a multiplicative identity $1$, where $R$ is not necessarily associative with respect to multiplication.
\end{defn}
\vspace{0.1cm}
When we assume associativity we will always use the phrasing \textit{associative rings}. 

\begin{defn}\label{ss}
	For a non-associative ring $R$, the \textit{associator }is the function $(\cdot,\cdot,\cdot): R \times R \times R \to R$ \label{associator} defined for all $r,s,t \in R$ as $(r,s,t) = (rs)t - r(st)$.
\end{defn}

\begin{defn}\label{ss1}
	Using the associator, we define the following sets,
	\[N_l(R) := \{ r \in R : (r,s,t) = 0 \text{ for all } s,t \in R \}\;\;\;\text{(\textit{left nucleus of R})}\label{lnuc},\] 
	\[N_m(R) := \{ s \in R : (r,s,t) = 0 \text{ for all } r,t \in R \}\;\;\;\text{(\textit{middle nucleus of R})} \label{mnuc},\]
	\[N_r(R) := \{ t \in R : (r,s,t) = 0 \text{ for all } r,s \in R \}\;\;\;\text{(\textit{right nucleus of R})} \label{rnuc}.\]
	The set $N_{l}(R)\cap N_{m}(R)\cap N_{r}(R)$ is called the \textit{nucleus} of $R$, denoted by $N(R)$\label{nuc}. The sets $N_l(R)$, $N_m(R)$ and $N_r(R)$ are associative subrings of $R$.
\end{defn}

%\begin{prop}
%	Let $R$ be a non-associative ring. For all elements $u,r,s,t\in R$, the following \textit{associator identity} holds,
%	\[u(r, s, t) + (u, r, s)t + (u, rs, t) = (ur, s, t) + (u, r, st).\]
%end{prop}

%\begin{rem}
%	By the associator identity, $N_{l}(R)$, $N_{m}(R)$, and $N_{r}(R)$ are associative subrings of $R$.
%\end{rem}

\begin{comment}
\begin{ex}\cite{Baez2002}
	\textnormal{Consider the set $\mathbb{O}$ of octonions over real numbers $\mathbb{R}$: \label{real}
	\[
	\mathbb{O}
	=
	\left\{
	a_0 1_{\mathbb{R}} + \sum_{i=1}^{7} a_i e_i
	\;\middle|\;
	a_0, a_1, \dots, a_7 \in \mathbb{R}
	\right\}.
	\]
	The set $\mathbb{O}$ \label{oct}is an 8-dimensional algebra on $\mathbb{R}$ with basis $\{1_{\mathbb{R}}, e_{1}, e_{2}, e_{3}, e_{4}, e_{5}, e_{6}, e_{7}\}$ where $1_{\mathbb{R}}$ is the multiplicative identity and $e_{i}^{2}=-1$ for each $1\leq i\leq 7$}.

    \begin{table}[H]
		\renewcommand{\thetable}{2.1}
		\centering
		\vspace*{-0.1cm}
		\includegraphics[height=6cm,width=8cm]{otable}
		\caption{Multiplication table for $\mathbb{O}$} % Optional: add a caption here
		\label{fig:otable} % Optional: add a label here
\end{table}

\textnormal{The octonions are not associative since there are basis elements, $e_{i}, e_{j}, e_{k}$, where the products $(e_{i}e_{j})e_{k}$ and $e_{i}(e_{j}e_{k})$ are additive inverses of each other. Consider the basis elements $e_{1}, e_{2},e_{3}\in\mathbb{O}$ where $(e_{1}e_{2})e_{3}=e_{4}e_{3}=-e_{6 }$ and $e_{1}(e_{2}e_{3})=e_{1}e_{5}=e_{6}$ so, $(e_{1},e_{2},e_{3})\neq0$.}
\end{ex}
\end{comment}

\vspace{0.1cm}
%If $A$ and $B$ are non-associative rings, then the direct product $A\times B$ with usual componentwise addition is an abelian group. The usual componentwise multiplication is also distributive over addition. $A\times B$ is not necessarily associative since for any $(a_{1},b_{1})$, $(a_{2},b_{2})$, $(a_{3},b_{3})\in A\times B$, $((a_{1},b_{1})(a_{2}.b_{2}))(a_{3},b_{3})\neq (a_{1},b_{1})((a_{2},b_{2})(a_{3},b_{3}))$ in general since for every component, $(a_{1}a_{2})a_{3}\neq a_{1}(a_{2}a_{3})$ and $(b_{1}b_{2})b_{3}\neq b_{1}(b_{2}b_{3})$ in general. This leads to the following remark
%\begin{rem}
 %   If $A$ and $B$ are non-associative rings, then $A\times B$ is also non-associative.
%\end{rem}

Now, we will introduce the ideal of a non-associative ring.
\vspace{0.1cm}
\begin{defn}\label{nideal}
	\cite{waliyanti2021nonassociative}
	Let $R$ be a non-associative ring. A  nonempty subset $I$ in $R$ is a \textit{left} (resp. right) \textit{ideal} of $R$ if $(I,+)$ is a subgroup of $R$ and for all $r\in R$, $rI:=\{ri\;|\;i\in I\}\subseteq I$ (resp. $Ir:=\{ir\;|\;i\in I\}\subseteq I$). If $I$ is both left and right ideal, then $I$ is called an \textit{ideal}.
\end{defn}
\vspace{0.1cm}
Note that Definition \ref{nideal} is similar to the ideal of an associative ring, but the concrete description is distinct. We adapt the paper of B\"ack and Richter \cite{Back2024} in 2024 for discussing the form of elements of a finitely generated right (resp. left) ideal of a non-associative ring.\\

In a non-associative ring $R$, the right (resp. left) ideal of $R$, generated by elements $a_{1},\ldots,a_{n}$ for some $n\in\mathbb{N}$, is not equal to the set of elements of the form $\left\{\sum\limits_{i=1}^{n}a_{i}r_{i}\;\;|\;\;r_{i}\in R\right\}$ \;\;\;$\Bigg($resp. $\left\{\sum\limits_{i=1}^{n}r_{i}a_{i}\;\;|\;\;r_{i}\in R\right\}\Bigg)$ which are the form of the elements when $R$ is associative. This is because, in general, $(a_{i}r_{i})r_{j}\neq a_{i}(r_{i}r_{j})$ \;\;(resp. $(r_{i}r_{j})a_{i}\neq r_{i}(r_{j}a_{i})$). Instead, the elements have the form
\begin{align*}
	&\sum_{i=1}^{n} \sum_{j=1}^{n_i} \sum_{l=1}^{k_{ij}} (\cdots((a_i r_{ijl1})r_{ijl2})\cdots)r_{ijlj}\\
	&\qquad\qquad\Bigg(\text{resp.}\;\;\sum_{i=1}^{n} \sum_{j=1}^{n_i} \sum_{l=1}^{k_{ij}} r_{ijlj}(\cdots(r_{ijl2}(r_{ijl1} a_i))\cdots)\Bigg)\qquad\text{(2.1)}
\end{align*}
for some $n_{i}, k_{ij}\in\mathbb{N}$ where 
$\sum\limits_{i=1}^{n}$ sums over the generators $a_{i}$, $\sum\limits_{j=1}^{n_i}$ sums over the length of the product for a specific $a_{i}$, that is, the number of times $a_{i}$ is being multiplied on the right (resp. left) by an element of $R$, and $\sum\limits_{l=1}^{k_{ij}}$ sums over different terms for a specific product length $j$ and generator $a_{i}$.\\

Now, let $X\subseteq R$. We apply the same concept above to the right (resp. left) ideal of $R$ generated by $X$. We denote this by $\langle X\rangle_{r}$ (resp. $\langle X\rangle_{l}$). $\langle X\rangle_{r}$ (resp. $\langle X\rangle_{l}$) consists of finite sums for any $n\in\mathbb{N}$ of the form given by (2.1).\\

We focus our attention on how Noetherianity is defined for non-associative rings. 
\begin{defn}\label{238}
	\cite{Back2024}
	Let $R$ be a non-associative ring. $R$ is \textit{left} (resp. \textit{right}) \textit{Noetherian}, if $R$ satisfies the ACC on left (resp. right) ideals, that is, for every ascending chain of left \textup{(}resp. right\textup{)} ideals $A_{1}\subset A_{2}\subset\cdots$, there is some $n\in\mathbb{N}$ such that $A_{i}=A_{n}$ for all $i\geq n$. If $R$ is both left and right Noetherian, then $R$ is \textit{Noetherian}.
\end{defn}
\vspace{0.1cm}
This produces the following remark.
\vspace{0.1cm}
\begin{rem}\label{239}
	\cite{Back2024}
	Let $R$ be a non-associative ring. $R$ satisfying ACC on left (resp. right) ideals, is equivalent to all left (resp. right) ideals being finitely generated.
\end{rem}
\begin{ex}\label{211}
	\textnormal{Consider the non-associative ring $\mathbb{O}$ of octonions. Since $\mathbb{O}$ is a division algebra \cite{Baez2002}, the only left and right ideals are $\{0\}$ and $\mathbb{O}$ itself. Hence, it easily satisfies the ACC for left and right ideals so $\mathbb{O}$ is Noetherian.}
\end{ex}
\vspace{0.1cm}
Now, we present non-associative skew power series rings $R[[X;\sigma]]$ and non-associative skew Laurent series rings $R((X;\sigma))$ whose definitions are adapted from \cite{Back2024}.
\vspace{0.1cm}
\begin{defn}\label{213}
	Let $R$ be a non-associative ring with an additive map $\sigma$ on $R$ such that $\sigma(1)=1$. A \textit{non-associative skew power series ring} $R[[X;\sigma]]$ \label{skewpowerseries} is a set of formal power series $
	\sum\limits_{i=0}^{\infty} r_i X^i$,
	where $r_i \in R$ with the usual pointwise addition and the multiplication is defined by $
	(r X^m)(s X^n):=(r \sigma^m(s)) X^{m+n}$
	for any $r, s \in R$ and $m, n \in \mathbb{N}$ which extends in the obvious way. 
\end{defn}

%The multiplication rule presented extends to any formal power series over $R$ as follows: Let $\sum\limits_{i=0}r_{i}X^{i},\sum\limits_{j=0}s_{j}X^{j}\in R[[X;\sigma]]$. Then
%\begin{align*}
 %   \left( \sum_{i=0}^{\infty} r_i X^i \right) \left( \sum_{j=0}^{\infty} s_j X^j \right)&=\sum_{i=0}^{\infty} \sum_{j=0}^{\infty} (r_i X^i)(s_j X^j)\\
  %  &=\sum_{i=0}^{\infty} \sum_{j=0}^{\infty}r_{i}\sigma^{i}(s_{j})X^{i+j}.
%\end{align*}
%We group each term of this sum by total powers of $X$. Let $k=i+j$. So, $j=k-i$ and since $j\geq0$, $i\leq k$. Then
%\[
%\sum_{i=0}^{\infty} %\sum_{j=0}^{\infty}r_{i}\sigma^{i}%(s_{j})X^{i+j}=\sum_{k=0}^{\infty}\left(\s%um_{i=0}^{k} r_i \sigma^i (s_{k-%i})\right)X^{k}.
%\]

\begin{ex}
	\textnormal{Consider $\mathbb{C}[[X;\sigma]]$ consisting of formal power series over $\mathbb{C}$ where $\mathbb{C}$ is a non-associative ring and $\sigma:\mathbb{C}\longrightarrow\mathbb{C}$ is defined by $\sigma(a+bi) = a-qbi$ for any $a,b,q \in \mathbb{R}$ where $q\neq\pm 1$. Together with the usual pointwise addition and multiplication defined by}
	\begin{align*}
		((a+bi) X^m)((c+di) X^n)&:=(a+bi) \;\sigma^m(c+di)) X^{m+n}\\
		&=\begin{cases} 
			(a+bi)(c+q^{m}di) & \text{if } m\;\;\text{is even} \\
			(a+bi)(c-q^{m}di)  & \text{if } m\;\;\text{is odd},
		\end{cases}
	\end{align*}
	\textnormal{this is a non-associative skew power series ring.}
\end{ex}
%\vspace{0.1cm}
%Particularly, the example above is not associative. Consider $X$ and $i$. Then
%\[(Xi)i = (qiX)i = qi(Xi) = qi(qiX) = %q^2i^2X = -q^2X\]
%\[X(ii) = X(-1) = \sigma(-1)X = -1X = -%X\] which can be equal only if $q=\pm 1$.
\vspace{0.1cm}
\begin{defn}\label{215}
	Let $R$ be a non-associative ring with an additive bijection $\sigma$ such that $\sigma(1)=1$. The \textit{non-associative skew Laurent series ring} $R((X;\sigma))$ \label{skewlaurentseries} is the set consisting of formal Laurent series $
	\sum\limits_{i=n}^{\infty} r_i X^i$, where $r_i \in R$, and $n \in \mathbb{Z}$, 
	with pointwise addition and the multiplication is determined by the same twisting rule which extends in the obvious way.
\end{defn}
%Similarly, the multiplication rule here also extends to any formal Laurent series over $R$ as follows: Let $\sum\limits_{i=n}r_{i}X^{i},\sum\limits_{j=m}s_{j}X^{j}\in R((X;\sigma))$ for some $n,m\in\mathbb{Z}$. Then
%\begin{align*}
 %   \left( \sum_{i=n}^{\infty} r_i X^i \right) \left( \sum_{j=m}^{\infty} s_j X^j \right)&=\sum_{i=n}^{\infty} \sum_{j=m}^{\infty} (r_i X^i)(s_j X^j)\\
  %  &=\sum_{i=n}^{\infty} \sum_{j=m}^{\infty}r_{i}\sigma^{i}(s_{j})X^{i+j}.
%\end{align*}
%Grouping each term of the sum in terms of total powers of $X$, let $k=i+j$ so $j=k-i$. Observe that $i\geq n$ and $j\geq m$ so $k-i\geq m$ which further implies $i\leq k-m$.
%So, 
%\[
%\sum_{i=n}^{\infty} \sum_{j=m}^{\infty}r_{i}\sigma^{i}(s_{j})X^{i+j}=\sum_{k=n+m}^{\infty}\left(\sum_{i=n}^{k-m} r_i \sigma^i (s_{j})\right)X^{k}.
%\]

\begin{ex}
\textnormal{Consider $\mathbb{C}((X;\sigma))$ consisting of formal Laurent series over $\mathbb{C}$ where $\mathbb{C}$ is a non-associative ring and $\sigma:\mathbb{C}\longrightarrow\mathbb{C}$ is defined by $\sigma(a+bi) = a-qbi$ for any $a,b,q \in \mathbb{R}$ where $q\neq\pm 1$. Together with the usual pointwise addition and the multiplication defined by}
\begin{align*}
	((a+bi) X^m)((c+di) X^n)&:=(a+bi) \;\sigma^m(c+di)) X^{m+n}\\
	&=\begin{cases} 
		(a+bi)(c+q^{m}di) & \text{if } m\;\;\text{is even} \\
		(a+bi)(c-q^{m}di)  & \text{if } m\;\;\text{is odd},
	\end{cases}
\end{align*}
\textnormal{this is a non-associative skew Laurent series ring.}
\end{ex}

Moreover, $R[[X;\sigma]]$ and $R((X;\sigma))$ are non-associative even if $R$ is associative, respectively. So, for any $r,s\in R$, $X(rs)$ is not necessarily equal to $(Xr)s$. This implies that $X$ and its powers are not necessarily in $N_{l}(R[[X;\sigma]])$ (resp. $N_{l}(R((X;\sigma))$). This leads to the following observation.

\begin{rem}\label{316}
    In $R[[X;\sigma]]$ (resp. $R((X;\sigma))$), \[X^{k}\in N_{m}(R[[X;\sigma]])\cap N_{r}(R[[X;\sigma]])\]\;\;\text{(resp. $X^{k}\in N_{m}(R((X;\sigma)))\cap N_{r}(R((X;\sigma)))$)}
 for any $k\in\mathbb{N}$ (resp. $k\in\mathbb{Z}$).
\end{rem}

\begin{defn}\label{ss}
	The \textit{order} of a non-zero element $p$ in either $R[[X;\sigma]]$ or $R((X;\sigma))$ is the least power of $X$ in $p$ with a non-zero coefficient, and that coefficient is the \textit{leading coefficient}.
\end{defn}
\vspace{0.1cm}
For simplicity, we denote the order of $p$ by $\mathrm{ord}(p)$ \label{ord} and the leading coefficient of $p$ by $\mathrm{lc}(p)$. \label{lc} Additionally, for any $f=\sum\limits_{i=0}^{\infty} r_i X^i\in R[[X;\sigma]]$ provided that $\mathrm{ord}(f)=n$ for some $n\in\mathbb{N}\cup\{0\}$ (resp. $ f=\sum\limits_{i=m}^{\infty} r_i X^i\in R((X;\sigma))$, $m\in\mathbb{Z}$ provided that $\mathrm{ord}(f)=m$), we can write it as
$f=r_{n}X^{n}+H.T.$ (resp. $f=r_{m}X^{m}+H.T.$) where $H.T.$ \label{HT} denotes higher terms of $f$.\\
\vspace{0.1cm}

\section{The Right and  Left Quasi-Associative Properties of a Non-associative Ring}\label{s3}

We introduce the following terminologies for a non-associative ring $R$ and these will be used later on developing the Hilbert Basis Theorem.

\begin{defn}\label{311}
    Let $R$ be a non-associative ring. $R$ is said to be \textit{right} (resp. \textit{left}) \textit{quasi-associative} if for all $r,b,c\in R$, there exists $s\in R$ such that $(rb)c=rs$ (resp. $r(bc)=sc$). If $R$ is both right and left quasi-associative, then $R$ is said to be \textit{quasi-associative}.
\end{defn}

If $R$ is associative, then it is trivially quasi-associative.
%\begin{rem}\label{32}
%	If $R$ is associative, then $R$ is trivially right (resp. left) quasi-associative because we can take $s:=bc$ (resp. $s:=rb$) so that $(rb)c=r(bc)=rs$ (resp. $r(bc)=(rb)c=sc$). In particular, $R$ is quasi-associative.
%\end{rem}
The following example shows that the right and left quasi-associative properties are not equivalent.

\begin{ex}\label{3111111}
 \textnormal{Let $R:=M_{2}(\mathbb{R})$ be the set of all $2\times2$ matrices whose entries are from $\mathbb{R}$ with usual matrix addition and multiplication $\cdot$ is defined as:\\ $A\cdot B:=((A\cdot B)_{ij})$ where} 
    \[(A\cdot B)_{11}:=A_{11}B_{11},\; (A\cdot B)_{12}:=A_{11}B_{12} + A_{12}B_{11}, \;(A\cdot B)_{21}:=A_{11}B_{21} + A_{21}B_{11} + A_{21}B_{12}\] \[(A\cdot B)_{22}:=A_{11}B_{22} + A_{22}B_{11} + A_{22}B_{21}\]
\textnormal{This is a non-associative ring where the new product satisfies left and right distributivity since it inherits the distributivity of $\mathbb{R}$ and the identity element is $E_{11}$ since for any $A=(a_{ij})$, $E_{11}\cdot A$ has  entries $(E_{11}\cdot A_{11})=1A_{11}=A_{11}$, $(E_{11}\cdot A)_{12}=1A_{12} + 0A_{11}=A_{12}$, $(E_{11}\cdot A)_{21}=1A_{21} + 0A_{11} + 0A_{12}=A_{21}$, and $(E_{11}\cdot A)_{22}=1A_{22} + 0A_{11} + 0A_{21}=A_{22}$ so $E_{11}\cdot A=A$. Likewise, $A\cdot E_{11}$ has entries $(A\cdot E_{11})_{11}=A_{11}1=A_{11}$, $(A\cdot E_{11})_{12}=A_{11}0 + A_{12}1=A_{12}$, $(A\cdot E_{11})_{21}=A_{11}0 + A_{21}1 + A_{21}0=A_{21}$, and $(A\cdot E_{11})_{22}=A_{11}0 + A_{22}1 + A_{22}0=A_{22}$ so $A\cdot E_{11}=A$. The new multiplication is not associative. To see this, let $A=E_{22}$, $B=E_{21}$, $C=E_{12}$. Then 
$A\cdot B$ has entries $(A\cdot B)_{11}=(0)(0)=0$, $(A\cdot B)_{12}=(0)(0) + (0)(0)=0$, $(A\cdot B)_{21}=(0)(1) + (0)(0) + (0)(0)=0$, and $(A\cdot B)_{22}=(0)(0) + (1)(0) + (1)(1)=1$. So, $A\cdot B=E_{22}$. Then
$(A\cdot B)\cdot C=E_{22}\cdot C$ where $E_{22}\cdot C$ has entries $(E_{22}\cdot C)_{11}=(0)(0)=0$, $(E_{22}\cdot C)_{12}=(0)(1) + (0)(0)=0$, $(E_{22}\cdot C)_{21}=(0)(0) + (0)(0) + (0)(1)=0$, and $(E_{22}\cdot C)_{22}=(0)(0) + (1)(0) + (1)(0)=0$. So, $(A\cdot B)\cdot C$ is a zero matrix. On the other hand, for $A\cdot (B\cdot C)$, the product $B\cdot C$ has entries $(B\cdot C)_{11}=(0)(0)=0$, $(B\cdot C)_{12}=(0)(1) + (0)(0)=0$, $(B\cdot C)_{21}=(0)(0) + (1)(0) + (1)(1)=1$, and $(B\cdot C)_{22}=(0)(0) + (0)(0) + (0)(0)=0$. So, $B\cdot C=E_{21}$. Then $A\cdot (B\cdot C)=A\cdot E_{21}$ where $A\cdot E_{21}$ has entries $(A\cdot E_{21})_{11}=(0)(0)=0$, $(A\cdot E_{21})_{12}=(0)(0) + (0)(0)=0$, $(A\cdot E_{21})_{21}=(0)(1) + (0)(0) + (0)(0)=0$, and $(A\cdot E_{21})_{22}=(0)(0) + (1)(0) + (1)(1)=1$. So, $A\cdot (B\cdot C)=E_{22}$. However, $(A\cdot B)\cdot C\neq A\cdot(B\cdot C)$.}

\textnormal{Now, we will show that $R$ satisfies right quasi-associativity but fails to be left quasi-associative. Let $A,B,C\in R$ and let $M\in R$ such that $M=(A\cdot B)\cdot C$. Let $P\in R$ such that $A\cdot B=P$. The entries of $P$ are $P_{11}=A_{11}B_{11}$, $P_{12}=A_{11}B_{12} + A_{12}B_{11}$, $P_{21}=A_{11}B_{21} + A_{21}(B_{11}+B_{12})$, and $P_{22}=A_{11}B_{22} + A_{22}(B_{11}+B_{21})$. Then $M=P\cdot C$ yields to a matrix with entries:}
\begin{align*}
    M_{11} &= A_{11}B_{11}C_{11} \\
M_{12} &= A_{11}(B_{11}C_{12} + B_{12}C_{11}) + A_{12}B_{11}C_{11} \\
M_{21} &= A_{11}(B_{11}C_{21} + B_{21}(C_{11}+C_{12})) + A_{21}(B_{11}+B_{12})(C_{11}+C_{12}) \\
M_{22} &= A_{11}(B_{11}C_{22} + B_{22}(C_{11}+C_{21})) + A_{22}(B_{11}+B_{21})(C_{11}+C_{21}).
\end{align*}
\textnormal{We construct $S\in R$ such that $A\cdot S=M$. The product $A\cdot S$ results in a matrix with entries:}
\begin{align*}
    (A\cdot S)_{11}&=A_{11}S_{11}\\
     (A\cdot S)_{12}&=A_{11}S_{12} + A_{12}S_{11}\\
      (A\cdot S)_{21}&=A_{11}S_{21} + A_{21}(S_{11}+S_{12})\\
       (A\cdot S)_{22}&=A_{11}S_{22} + A_{22}(S_{11}+S_{21}).
\end{align*}
\textnormal{We analyze two cases: If $A_{11}\neq 0$, then the entries of $S$ are:}
\begin{align*}
    S_{11}&=B_{11}C_{11}\\
    S_{12}&=B_{11}C_{12} + B_{12}C_{11}\\
    S_{21}&=B_{11}C_{21} + B_{21}(C_{11}+C_{12}) + \frac{A_{21}}{A_{11}}B_{12}C_{12}\\
    S_{22}&=B_{11}C_{22} + B_{22}(C_{11} + C_{21}) + \frac{A_{22}((B_{11} + B_{21})(C_{11} + C_{21}) - (B_{11}C_{11} + S_{21}))}{A_{11}}\\
\end{align*}
\textnormal{If $A_{11}=0$, then the entries of $M$ simplifies  to:}
\[
M_{11} = 0, \quad M_{12} = A_{12}B_{11}C_{11}, \quad M_{21} = A_{21}(B_{11}+B_{12})(C_{11}+C_{12})
\]
\[
 M_{22} = A_{22}(B_{11}+B_{21})(C_{11}+C_{21})
\]
\textnormal{Define $S\in R$ with entries}
\[
S_{11}= B_{11}C_{11}, \quad S_{12}= (B_{11}+B_{12})(C_{11}+C_{12}) - B_{11}C_{11},
\]
\[
S_{21}= (B_{11}+B_{21})(C_{11}+C_{21}) - B_{11}C_{11}, \quad S_{22}= 0.
\]
\textnormal{Using this matrix $S$, the product $A\cdot S$ has entries}
\begin{align*}
(A \cdot S)_{11} &= A_{11}S_{11} \\
&= (0)S_{11} =0\\[8pt]
(A \cdot S)_{12} &= A_{11}S_{12} + A_{12}S_{11} \\
&= (0)S_{12} + A_{12}(B_{11}C_{11}) \\
&= A_{12}B_{11}C_{11} \\[8pt]
(A \cdot S)_{21} &= A_{11}S_{21} + A_{21}(S_{11} + S_{12}) \\
&= (0)S_{21} + A_{21}( B_{11}C_{11} + ((B_{11} + B_{12})(C_{11} + C_{12}) - B_{11}C_{11})) \\
&= A_{21}( B_{11}C_{11} - B_{11}C_{11} + (B_{11} + B_{12})(C_{11} + C_{12})) \\
&= A_{21}(B_{11} + B_{12})(C_{11} + C_{12}) \\[8pt]
(A \cdot S)_{22} &= A_{11}S_{22} + A_{22}(S_{11} + S_{21}) \\
&= (0)(0) + A_{22}( B_{11}C_{11} + ((B_{11} + B_{21})(C_{11} + C_{21}) - B_{11}C_{11})) \\
&= A_{22}( B_{11}C_{11} - B_{11}C_{11} + (B_{11} + B_{21})(C_{11} + C_{21})) \\
&= A_{22}(B_{11} + B_{21})(C_{11} + C_{21}).
\end{align*}
\textnormal{which coincides with each entries of $M$. For each cases, we have shown that there exists $S\in R$ satisfying right quasi-associativity. On the other hand,  suppose $R$ satisfies left quasi-associativity. Let $A=E_{22}$, $B=E_{21}$, $C=E_{12}$. From earlier, $A\cdot (B\cdot C)=E_{22}$. Let $S\in R$ satisfying left quasi-associativity. The product $S\cdot C$ results to a matrix whose has entries are}
\begin{align*}
    (S\cdot C)_{11}&=S_{11}(0)=0\\
    (S\cdot C)_{12}&=S_{11}(1) + S_{12}(0)=S_{11}\\
    (S\cdot C)_{21}&=S_{11}(0) + S_{21}(0) + S_{21}(1)=S_{21}\\
    (S\cdot C)_{22}&=S_{11}(0) + S_{22}(0) + S_{22}(0)=0.
\end{align*}
\textnormal{Since $R$ is left quasi-associative by assumption, $A\cdot(B\cdot C)=E_{22}=S\cdot C$. But observe that $(S\cdot C)_{22}=0\neq 1=(E_{22})_{22}$ which is a contradiction. This shows that there is no $S\in R$ that satisfies left quasi-associativity.}
\end{ex}

\textnormal{Equivalently, Definition \ref{311} means that $(rb)c\in rR:=\{rt\;|\;t\in R\}$ (resp. $r(bc)\in Rc:=\{tc\;|\;t\in R\}$). Now, we have the following characterization for a non-associative ring $R$ to be right (resp. left) quasi-associative.}

\begin{prop}\label{3113}
    Let $R$ be a non-associative ring.  $R$ is right (resp. left) quasi-associative if and only if $rR$ (resp. $Rc$) is a right (resp. left) ideal of $R$ for all $r\in R$ (resp. $c\in R$). 
\end{prop}

%\begin{proof}
 %   Let $r,b,c\in R$. Note that $rR$ is nonempty since $r0=0\in rR$. Let $s,t\in rR$. This means $s=rp$ and $t=rq$ for some $p,q\in R$. Then $s-t=rp-rq=r(p-q)\in rR$ where $p-q\in R$. On the other hand, observe that $rb\in rR$. Suppose $R$ satisfies right quasi-associativity. Then $(rb)c=rs$ for some $s\in R$, which implies $(rb)c\in rR$, showing the right absorption law. Thus, $rR$ is a right ideal of $R$.

  %  Conversely, assume $rR$ is a right ideal of $R$. Note that $rb\in rR$. By the right absorption law, $(rb)c\in rR$. Therefore, $R$ satisfies right quasi-associativity.

   % The analogous argument follows for left quasi-associativity.
%\end{proof}

\begin{prop}\label{311116}
    If $R$ is a non-associative division ring, then $R$ is quasi-associative.
\end{prop}

\begin{proof}
     Let $R$ be a non-associative division ring. Let $r\in R$. If $r=0$, then $rR=0R=\{0\}$ which is clearly a right ideal. If $r\neq 0$, then for every $s\in R$, the equation $rx=s$ has a unique solution $x=r^{-1}s\in R$. So, $s\in rR$ implying that $R\subseteq rR$. The reverse inclusion is obvious, and so $rR=R$, which is clearly a right ideal. Hence, by Proposition \ref{3113}, $R$ is right quasi-associative. Analogously, the set $Rc$ is a left ideal of $R$ for every $c\in R$ to conclude that $R$ is left quasi-associative. Thus, $R$ is quasi-associative.
\end{proof}

\begin{ex}\label{3114}
     \textnormal{Since $\mathbb{O}$ is a division algebra \cite{Baez2002}, by Proposition \ref{311116}, $\mathbb{O}$ is quasi-associative.}
\end{ex}

However, the converse of Proposition \ref{311116} is false even when $R$ is associative. On the other hand, the right and left quasi-associativity are also investigated on direct products of non-associative rings. 
%If $A$ and $B$ are right (resp. left) quasi-associative, then $A\times B$ is also right (resp. left) quasi-associative since $A\times B$ inherits the right (resp. left) quasi-associativity in a componentwise manner under multiplication. 
We have the following remark,
\begin{rem}\label{rem}
    If $A$ and $B$ are right (resp. left) quasi-associative, then $A\times B$ is also right (resp. left) quasi-associative.
\end{rem}
\vspace{0.1cm}

We also explored the behavior of right (resp. left) quasi-associativity between a non-associative coefficient ring $R$ and its corresponding non-associative skew power series rings $R[[X;\sigma]]$ and non-associative skew Laurent series rings $R((X;\sigma))$.

\begin{prop}\label{inherit}
    Let $R$ be a non-associative ring and $\sigma:R\longrightarrow  R$ be an additive map such that $\sigma(1)=1$. If $R[[X;\sigma]]$ is right (resp. left) quasi-associative, then $R$ is also right (resp. left) quasi-associative.
\end{prop}

\begin{proof}
    Let $r,b,c\in R$. Suppose $R[[X;\sigma]]$ is right quasi-associative. Then there exists $f=\sum\limits_{i=0}^{\infty}r_{i}X^{i}\in R[[X;\sigma]]$ such that $(rb)c=rf=r\left(\sum\limits_{i=0}^{\infty}r_{i}X^{i}\right)=\sum\limits_{i=0}^{\infty}(rr_{i})X^{i}$. Since $r,b,c$ are viewed as constant formal power series, respectively, in $R[[X;\sigma]]$, they have order $0$. So, $\mathrm{ord}((rb)c)=0=\mathrm{ord}\left(\sum\limits_{i=0}^{\infty}(rr_{i})X^{i}\right)$. This implies $(rb)c=rr_{0}$ for some $r_{0}\in R$. Thus, $R$ is  right quasi-associative.

For the left version, suppose $R[[X;\sigma]]$ is left quasi-associative. Then there exists $g=\sum\limits_{i=0}^{\infty}s_{i}X^{i}\in R[[X;\sigma]]$ such that $r(bc)=gc=\left(\sum\limits_{i=0}^{\infty}s_{i}X^{i}\right)c=\sum\limits_{i=0}^{\infty}s_{i}X^{i}c=\sum\limits_{i=0}^{\infty}s_{i}\sigma^{i}(c)X^{i}$. By the analogous argument in the right version, this implies $r(bc)=s_{0}\sigma^{0}(c)=s_{0}c$ for some $s_{0}\in R$. Hence, $R$ is left quasi-associative.
\end{proof}

Using the analogous argument from Proposition \ref{inherit}, we have the following observation for $R((X;\sigma))$.

\begin{prop}\label{013}
       Let $R$ be a non-associative ring and $\sigma:R\longrightarrow  R$ be an additive bijection such that $\sigma(1)=1$. If $R((X;\sigma))$ is right (resp. left) quasi-associative, then $R$ is also right (resp. left) quasi-associative.
\end{prop}

However, the converse of Proposition \ref{013} is false.
\begin{ex}\label{0013}
    Let $R := M_2(\mathbb{R})$ with usual addition and multiplication of matrices. Since $R$ is associative, then $R$ is trivially right quasi-associative. Let $\sigma: R\longrightarrow R$ be the transpose map, that is, $\sigma(A)=A^{T}$ for every $A\in R$ where $A^{T}$ is the transpose matrix of $A$. This is obviously an additive bijection map on $R$ which preserves the identity matrix $I$, that is, $\sigma(I)=I^{T}=I$.

    Consider the non-associative skew Laurent series ring $R((X;\sigma))$ and elements $r=E_{11}+E_{22}X$, $b=E_{11}+E_{12}$, $c=E_{12}\in R((X;\sigma))$. Then
    \[
    (rb)c=(E_{11}b+E_{22}b^{T}X)(E_{12})=E_{11}bc+E_{22}b^{T}c^{T}X=E_{12}
    \]
    where $E_{11}bc=E_{12}$ and $E_{22}b^{T}c^{T}$ is the zero matrix. Suppose $R((X;\sigma))$ is right quasi-associative, then there exists some $f=\sum\limits_{i=n}^{\infty}r_{i}X^{i}\in R((X;\sigma))$ for some $n\in\mathbb{Z}$ such that $(rb)c=rf$. This implies that $\mathrm{ord}(rf)=0$. Particularly, the product $rf$ is
    \[
    rf=\sum\limits_{i=n}^{\infty}(E_{11}r_{i}X^{i}+E_{22}r_{i}^{T}X^{i+1})
    \]
    For $i=0$, the corresponding term is given by $E_{11}r_{0}+E_{22}r_{0}^{T}X$.
    Since $(rb)c=rf$, then $E_{12}=E_{12}+0X=E_{11}r_{0}+E_{22}r_{0}^{T}X$. Note that $r_{0}:=E_{12}$ so that $E_{11}r_{0}=E_{12}$. Getting the transpose, $r_{0}^{T}=E_{21}$. Now, the product $E_{22}E_{21}=E_{21}$ which is not the zero matrix. This is a contradiction. Thus, $R((X;\sigma))$ is not right quasi-associative.

    For the left version, consider the elements $r=E_{12}X^{-1}$, $b=E_{21}$, $c=E_{11}+E_{12}X\in R((X;\sigma))$. Note that the inverse map of $\sigma$, $\sigma^{-1}$ is also the transpose map on $R$. Then
    \begin{align*}
    r(bc)=E_{12}X^{-1}(bE_{11}+bE_{12}X)=E_{12}X^{-1}(E_{21}+E_{22}X)&=E_{12}E_{12}X^{-1}+E_{12}E_{22}\\
    &=E_{12}.
    \end{align*}
    Suppose $R((X;\sigma))$ is left quasi-associative, then there exists some $g=\sum\limits_{i=m}^{\infty}r_{i}X^{i}\in R((X;\sigma))$ for some $m\in\mathbb{N}$ such that $r(bc)=gc$. This implies that $\mathrm{ord}(gc)=0$. Particularly, the product $gc$ is
    \[
    gc=\sum\limits_{i=m}^{\infty}(r_{i}E_{11}X^{i}+r_{i}\sigma^{i}(E_{12})X^{i+1})
    \]
    For $i=0$, the corresponding term is $r_{0}E_{11}+r_{0}E_{12}X$. Since $r(bc)=gc$, then $E_{12}=r_{0}E_{11}+r_{0}E_{12}X$. However, no such $r_{0}\in R$ satisfy this equation because $(r_{0}E_{11})_{12}=0\neq 1=(E_{12})_{12}$. This is a contradiction to the existence of $g$. Thus, $R((X;\sigma))$ is not left quasi-associative.
\end{ex}
 Example \ref{0013} also shows that the converse of Proposition \ref{inherit} is false by considering $R[[X;\sigma]]$ and applying the similar argument.
\vspace{0.1cm}

\section{The Hilbert Basis Theorem for Non-associative Skew Power Series Rings and Non-associative Skew Laurent Series Rings}

 We are now ready to prove the Hilbert Basis Theorem for right Noetherianity of $R[[X;\sigma]]$.
\vspace{0.1cm}
\begin{thm}\label{314}
   Let $R$ be a non-associative ring and $\sigma:R\longrightarrow R$ be an additive map such that $\sigma(1)=1$. If $R$ is a right Noetherian ring that is right quasi-associative and $\sigma$ is a surjective map, then $R[[X; \sigma]]$ is right Noetherian.
\end{thm}

\begin{proof}

 Let $P$ be a nonzero right ideal of $R[[X;\sigma]]$. Define the set
	\begin{equation*}
		Q=:\{0\}\cup \{\mathrm{lc}(f)\;|\;f\in P\setminus\{0\}\}.
	\end{equation*}
	\textbf{Claim}: $Q$ is a right ideal of $R$.
	
	We show that $Q$ is an additive subgroup of $R$. Let $a,b\in Q$, with $a\neq b$. Then there exists $f,g\in P$ such that $a=\mathrm{lc}(f)$ and $b=\mathrm{lc}(g)$. Here, $f-g\in P$. If $\mathrm{ord}(f)=\mathrm{ord}(g)$, then $\mathrm{lc}(f-g)=a-b$ and so $a-b\in Q$. On the other hand, if $\mathrm{ord}(f)\neq\mathrm{ord}(g)$, then consider elements $fX^{\mathrm{ord}(g)}$ and $gX^{\mathrm{ord}(f)}$. These are elements of $P$ and so is $fX^{\mathrm{ord}(g)}-gX^{\mathrm{ord}(f)}$ where
	\begin{equation*}
		\mathrm{ord}(fX^{\mathrm{ord}(g)})=\mathrm{ord}(f)+\mathrm{ord}(g)=\mathrm{ord}(gX^{\mathrm{ord}(f)})
	\end{equation*}
	So the leading term of the difference is $(a-b)X^{\mathrm{ord}(f)+\mathrm{ord}(g)}$. Hence, $a-b\in Q$. On the other hand, let $r\in R$, $q\in Q$. Then there exists $h\in P$ such that $q=\mathrm{lc}(h)$.  Now, observe that $\sigma$ is surjective, so $\sigma^{\mathrm{ord}(h)}$ is also surjective. This implies that there is some $r'\in R$ such that $\sigma^{\mathrm{ord}(h)}(r')=r$. Consider the element $hr'$. Since $P$ is a right ideal, $hr'\in P$ and
	\begin{align*}
		hr'&=(qX^{\mathrm{ord}(h)}+ H.T.)\cdot r'\\
		&=qX^{\mathrm{ord}(h)}r'+H.T.\\
		&=q\sigma^{\mathrm{ord}(h)}(r')X^{\mathrm{ord}(h)}+H.T.\\
		&=qrX^{\mathrm{ord}(h)}+H.T..
	\end{align*}
	Hence, $\mathrm{lc}(hr')=qr\in Q$ and so $Q$ is a right ideal of $R$.
	
	Now, let $f_{1}\in P$ be a nonzero element such that $f_{1}$ has the least order among non-zero elements of $P$ and define $q_{1}:=\mathrm{lc}(f_{1})$. Next, choose $f_{2}\in P$ such that $f_{2}$ has the least order among all the elements of $P$ whose leading coefficient is not in $\langle q_{1} \rangle_{r}$ and set $q_{2}:=\mathrm{lc}(f_{2})$. Continue the choice inductively, that is, choose $f_{i}\in P$ such that $f_{i}$ has the least order among all the elements of $P$ whose leading coefficient is not in $\langle q_{1},\ldots,q_{i-1} \rangle_{r}$ and set $q_{i}:=\mathrm{lc}(f_{i})$. This process creates an ascending chain of right ideals,
   \begingroup
\setlength{\abovedisplayskip}{7pt}
\setlength{\belowdisplayskip}{7pt}
\begin{align*}
    I_{1}&=\langle q_{1} \rangle_{r}\subset 
    I_{2}=\langle q_{1},q_{2} \rangle_{r}\subset\cdots\subset I_{i}=\langle q_{1},q_{2},\ldots,q_{i} \rangle_{r}\subset\cdots
\end{align*}
\endgroup
		Note that $Q$ consists of all leading coefficients of elements of $P$ so it is a part of the chain. Since $R$ is right Noetherian the chain stabilizes. So, the choice process on $f_{i}$ must stop and so $Q=I_{n}$ for some $n\in\mathbb{N}$. Set $A:=\{f_{1},\ldots,f_{n}\}$ and $d_{i}:=\mathrm{ord}(f_{i})$.\\

	    \textbf{Claim}: $A$ generates $P$\\
        \vspace{0.2cm}
        Let $p\in P$ where $\mathrm{ord}(p)=u$ and $\mathrm{lc}(p)=v$. Define $p_{0}:=p$. Clearly $v\in Q$ and since $R$ is right quasi-associative there exists $\beta_{i,0}'\in R$ such that
         
     \[
      v=\sum\limits_{i=1}^{n}q_{i}\beta'_{i,0}.
     \]

     Note that if $u<d_a$ for some minimum $a$ then $\beta'_{i,0}=0$ for all $i\geq a$ by the choice of $A$. 
     
     Since $\sigma$ is surjective, $\sigma^{d_{i}}$ is also surjective so there exists $\beta_{i,0}\in R$ such that $\sigma^{d_{i}}(\beta_{i,0})=\beta_{i,0}'$. Consider the element 
     \[
     s_0 := \sum_{i=1}^{n} f_i \beta_{i,0} X^{u - d_i}\in P.
     \]
     Notice that $\mathrm{ord}(s_{0})=u=\mathrm{ord}(p)$ and $\mathrm{lc}(s_{0})=\sum\limits_{i=1}^{n}q_{i}\sigma^{d_{i}}(\beta_{i,0})=\sum\limits_{i=1}^{n}q_{i}\beta'_{i,0}=v=\mathrm{lc}(p)$. Set $p_1 := p_0 - s_0 \in P$. Since $\mathrm{lc}(p_0) = v_0 = \mathrm{lc}(s_0)$, $p_1$ is either $0$ or $\mathrm{ord}(p_0) < \mathrm{ord}(p_1)$. Repeating this process inductively on $p_{m}\in P$ yields a sequence of elements in $P$ of the form $s_{m}:=\sum\limits_{i=1}^{n} f_i \beta_{i,m} X^{\mathrm{ord}(p_{m}) - d_i}$ where $p_{m}:=p_{m-1}-s_{m-1}$ is either $0$ or $\mathrm{ord}(p_{m-1})<\mathrm{ord}(p_{m})$. Observe that $p=(p_{0}-p_{1})+(p_{1}-p_{2})+(p_{2}-p_{3})+\cdots=\sum\limits_{m=0}^{\infty}s_{m}$so
	 \begin{align*}
p=\sum\limits_{m=0}^{\infty}s_{m}&=\sum\limits_{m=0}^{\infty}\Bigg(\sum\limits_{i=1}^{n}f_{i}\beta_{i,m}X^{\mathrm{ord}(p_{m})-d_{i}}\Bigg)\\
&=\sum\limits_{i=1}^{n}\Bigg(\sum\limits_{m=0}^{\infty}f_{i}\beta_{i,m}X^{\mathrm{ord}(p_{m})-d_{i}}\Bigg)\\
&=\sum\limits_{i=1}^{n}f_{i}\Bigg(\sum\limits_{m=0}^{\infty}\beta_{i,m}X^{\mathrm{ord}(p_{m})-d_{i}} \Bigg).
\end{align*}
	 In this case, $p$ can be written as a finite linear combination of elements of $A$. Hence, $P$ is a finitely generated right ideal of $R[[X;\sigma]]$ and so $R[[X;\sigma]]$ is right Noetherian.
\end{proof}

Theorem \ref{314} recovers the Hilbert Basis Theorem for right Noetherianity of skew power series rings from \cite{Goodearl} whenever $R$ is associative, right Noetherian, and $\sigma$ is an automorphism on $R$. Likewise, it generalizes the version of the theorem for non-associative skew power series rings over an associative coefficient ring established in \cite{Back2024}.

\begin{rem}

If $R$ is a right Noethrian non-associative ring that is not right quasi-associative and $P$ is a right ideal of $R[[X;\sigma]]$, then we can still find $f_1, \ldots f_n$ like in the proof of Theorem \ref{314} but we can only say that if $p \in P$,
then 
\[
p=\sum\limits_{i=1}^{n}\sum\limits_{j=1}^{n_{i}}\sum\limits_{l=1}^{k_{ij}}\Bigg(\sum\limits_{m=0}^{\infty}(\cdots((f_{i}t_{ijl1_{m}}')t_{ijl2_{m}}')\cdots)t_{ijlj_{m}}')X^{m}\Bigg)
\]
for some $n_{i}, k_{ij}\in\mathbb{N}$, $t'_{ijl1_{m}}, t'_{ijl2_{m}},\ldots,t'_{ijlj_{m}}\in R$.

This does not show that $P$ is a finitely generated right ideal of $R[[X;\sigma]]$ since an infinite sum of elements in a right ideal need not lie in the right ideal even when the sum is defined in $R[[X;\sigma]]$; in this case, the sum
\[
\sum\limits_{m=0}^{\infty}(\cdots((f_{i}t_{ijl1_{m}}')t_{ijl2_{m}}')\cdots)t_{ijlj_{m}}')X^{m}.
\]
Factoring $f_{i}$ to the left out of the infinite sum is not guaranteed since $f_{i}$ is not necessarily in $N_{l}(R[[X;\sigma]])$. 

\end{rem}
\vspace{0.1cm}
We have an example of a non-associative skew power series ring $R[[X;\sigma]]$ that is right Noetherian.
\vspace{0.1cm}
\begin{ex}
	\textnormal{Consider $\mathbb{O}$ and $\mathbb{R}$ as vector spaces over $\mathbb{Q}$. Consider a Hamel Basis $H$ of $\mathbb{R}$ over $\mathbb{Q}$, that is, a set consisting of linearly independent elements where every element of $\mathbb{R}$ can be uniquely expressed as a linear combination using finitely many elements from this set. Let $B=\{b_{0}:=1_{\mathbb{R}},b_{1},b_{2},\ldots,\}$ be a denumerable subset of $H$.
	Define a $\mathbb{Q}$-linear map $\psi:\mathbb{R}\longrightarrow\mathbb{R}$ by $\psi(1_{\mathbb{R}})=1_{\mathbb{R}}$, $\psi(b_{1})=0$, $\psi(b_{n})=b_{n-1}$ for all $n\geq 2$, and $\psi(h)=h$ for all $h\in H\setminus B$. Let $s=\sum\limits_{i=0}^{7}a_{i}e_{i}\in\mathbb{O}$. Define $\sigma:\mathbb{O}\longrightarrow\mathbb{O}$ to be $\sigma(s):=\sum\limits_{i=0}^{7}\psi(a_{i})e_{i}$.
	Here, $\sigma$ is additive since $\psi$ is additive. Also, $\sigma$ is surjective because if we let $u\in\mathbb{R}$, then by definition of $H$, there is some $m,k\in\mathbb{N}$ such that}
	\[
	u=\sum_{\substack{i=0\\ b_{i}\in B}}^{m}\beta_{i}b_{i}+\sum_{\substack{j=0\\ h_{j}\in H\setminus B}}^{k} \gamma_{j}h_{j}
	\]
	\textnormal{for some $\beta_{0},\ldots,\beta_{m}, \gamma_{0},\ldots,\gamma_{k}\in\mathbb{Q}$. Consider the element} \[u'=\beta_{0}b_{0}+\sum_{\substack{i=1\\ b_{i}\in B}}^{m}\beta_{i}b_{i+1}+\sum_{\substack{j=0\\ h_{j}\in H\setminus B}}^{k} \gamma_{j}h_{j}\]
	\textnormal{of $\mathbb{R}$. Since $\psi$ is $\mathbb{Q}$-linear map,}
	\begin{align*}
		\psi(u')&=\psi\left( \beta_{0}b_{0}+\sum_{\substack{i=1\\ b_{i}\in B}}^{m}\beta_{i}b_{i+1}+\sum_{\substack{j=0\\ h_{j}\in H\setminus B}}^{k} \gamma_{j}h_{j}\right)\\
		&=\psi(\beta_{0}b_{0})+\psi\left(  \sum_{\substack{i=1\\ b_{i}\in B}}^{m}\beta_{i}b_{i+1}   \right) +\psi\left(  \sum_{\substack{j=0\\ h_{j}\in H\setminus B}}^{k} \gamma_{j}h_{j}\right)\\
		&=\psi(\beta_{0}b_{0})+\sum_{\substack{i=1\\ b_{i}\in B}}^{m}\psi(\beta_{i}b_{i+1})+ \sum_{\substack{j=0\\ h_{j}\in H\setminus B}}^{k} \psi(\gamma_{j}h_{j})\\
		&=\beta_{0}b_{0}+\sum_{\substack{i=1\\ b_{i}\in B}}^{m}\beta_{i}\psi(b_{i+1})+\sum_{\substack{j=0\\ h_{j}\in H\setminus B}}^{k} \gamma_{j}\psi(h_{j})\\
		&=\sum_{\substack{i=0\\ b_{i}\in B}}^{m}\beta_{i}b_{i}+\sum_{\substack{j=0\\ h_{j}\in H\setminus B}}^{k} \gamma_{j}h_{j}=u.
	\end{align*}
	\textnormal{By definition of $\sigma$, it inherits the surjectivity of $\psi$. Moreover, $\sigma(1_{\mathbb{O}})=1_{\mathbb{O}}$ since $\psi(1_{\mathbb{R}})=1_{\mathbb{R}}$ where $1_{\mathbb{R}}=1_{\mathbb{O}}$ and by definition of $\sigma$. Furthermore, $\mathbb{O}$ is right Noetherian since it is a division algebra so the only right ideals are $\{0\}$ and $\mathbb{O}$ itself which satisfies the ACC on right ideals obviously. From Example \ref{3114}, $\mathbb{O}$ is right quasi-associative. Consider the non-associative skew power series ring $\mathbb{O}[[X;\sigma]]$. By Theorem \ref{314}, $\mathbb{O}[[X;\sigma]]$ is right Noetherian.}
\end{ex}
\vspace{0.1cm}

	Now, we prove the left version of Hilbert Basis Theorem of $R[[X;\sigma]]$.
    
\begin{thm}\label{3110}
		  Let $R$ be a non-associative ring and $\sigma:R\longrightarrow R$ be an additive map such that $\sigma(1)=1$. If $R$ is a left Noetherian ring that is left quasi-associative and $\sigma$ is a bijective map, then $R[[X;\sigma]]$ is left Noetherian.
	\end{thm}
	\begin{proof}
Let $R$ be a non-associative left Noetherian ring and $\sigma:R\longrightarrow R$ be an additive bijection map such that $\sigma(1)=(1)$. Observe that for any $d=\sum\limits_{i=0}^{\infty}a_{i}X^{i}\in R[[X;\sigma]]$, it can be written with right-hand coefficients by using $a_{i}X^{i}=X^{i}\sigma^{-i}(a_{i})$ in $R[[X;\sigma]]$.  So
		\begin{align*}
			d=\sum\limits_{i=0}^{\infty}a_{i}X^{i}&=a_{0}+X\sigma^{-1}(a_{1})+X^{2}\sigma^{-2}(a_{2})+\cdots\\
			&=\sum\limits_{i=0}^{\infty}X^{i}\sigma^{-i}(a_{i})
		\end{align*}
		Each $\sigma^{-i}(a_{i})$ are the coefficients corresponding to $a_{i}$ and $\sigma^{\mathrm{-ord}(d)}(\mathrm{lc}(d))$ is the coefficient corresponding to $\mathrm{lc}(d)$. Also, for any $X^{i}\sigma^{-i}(a_{i})$, $X^{j}\sigma^{-j}(b_{j})\in R[[X;\sigma]]$, the product is
		\begin{align*}
			(X^{i}\sigma^{-i}(a_{i}))(X^{j}\sigma^{-j}(b_{j}))&=(\sigma^{i}(\sigma^{-i}(a_{i})X^{i})(\sigma^{j}(\sigma^{-j}(b_{j}))X^{j})\\
			&=(a_{i}X^{i})(b_{j}X^{j})\\
			&=a_{i}\sigma^{i}(b_{j})X^{i+j}\\
			&=X^{i+j}\sigma^{-(i+j)}(a_{i}\sigma^{i}(b_{j}))
		\end{align*}
		and extend in the obvious way to general elements of $R[[X;\sigma]]$. 
		
		Let $P$ be a nonzero left ideal of $R[[X;\sigma]]$. Define the set $Q$,
		\[
			Q:=\{0\}\cup\{\sigma^{\mathrm{-ord}(f)}(\mathrm{lc}(f))\}\;|\;f\in P\setminus\{0\}\}.
		\]
		\textbf{Claim}: $Q$ is a left ideal of $R$.\\
		We will show that $Q$ is an additive subgroup of $R$. Let $a,b\in Q$. Then there exists $f,g\in P$ such that $a=\sigma^{\mathrm{-ord(f)}}(\mathrm{lc}(f))$ and $b=\sigma^{\mathrm{-ord(g)}}(\mathrm{lc}(g))$. Suppose $\mathrm{ord}(f)=\mathrm{ord}(g)$. In this case, define $b:=\mathrm{ord}(f)=\mathrm{ord}(g)$. Then
		\begin{align*}
			a-b&=(\sigma^{-b}(\mathrm{lc}(f)))-(\sigma^{\mathrm{-b}}(\mathrm{lc}(g)))\\
			&=\sigma^{-b}(\mathrm{lc}(f)-\mathrm{lc}(g))\\
			&=\sigma^{-b}(\mathrm{lc}(f-g))\in Q
		\end{align*}
		with $f-g\in P$. On the other hand, suppose that $\mathrm{ord}(f)\neq\mathrm{ord}(g)$. WLOG, suppose $\mathrm{ord}(f)<\mathrm{ord}(g)$. Consider the element $g_{1}:=X^{\mathrm{ord}(f)-\mathrm{ord}(g)}g$. Obviously, $g_{1}\in P$. Observe that
		\begin{align*}
			g_{1}=X^{\mathrm{ord}(f)-\mathrm{ord}(g)}g&=X^{\mathrm{ord}(f)-\mathrm{ord}(g)}(X^{\mathrm{ord}(g)}b+H.T)\\
			&=X^{\mathrm{ord}(f)}b+H.T.
		\end{align*}
		This makes $\mathrm{ord}(g_{1})=\mathrm{ord}(f)$. Also, $\mathrm{lc}(g)=\mathrm{lc}(g_{1})$, i.e., $g$ and $g_{1}$ represent the same $b$. So, we can just assume WLOG that $\mathrm{ord}(g)=\mathrm{ord}(f)$. Denote this order by $z$. Now, for this case, we have $a=\sigma^{-z}(\mathrm{lc}(f))$ and $b=\sigma^{-z}(\mathrm{lc}(g))$. Then
		\begin{align*}
			f-g&=(X^{z}a+H.T.)-(X^{z}b+H.T.)\\
			&=X^{z}(a-b)+H.T.\\
			&=X^{z}(\sigma^{-z}(\mathrm{lc}(f))-\sigma^{-z}(\mathrm{lc}(g)))+H.T.\\
			&=X^{z}\sigma^{-z}(\mathrm{lc}(f)-\mathrm{lc}(g))+H.T.\\
			&=X^{z}\sigma^{-z}(\mathrm{lc}(f-g))+H.T.
		\end{align*}
		with $f-g\in P$. Hence, $a-b=\sigma^{-z}(\mathrm{lc}(f-g))\in Q$.

		 Now, let $q\in Q$ and let $r\in R$. If $q\in Q$, then there exists $h\in P$ such that $h=\mathrm{lc}(h)X^{\mathrm{ord}(h)}+H.T.$ where $q=\sigma^{-\mathrm{ord}(h)}(\mathrm{lc}(h))$. Written with right-hand coefficients, we have $h=X^{\mathrm{ord}(h)}q+H.T.$. Consider $\sigma^{\mathrm{ord}(h)}(r)h$ which is an element of $P$. Then
		\begin{align*}
			\sigma^{\mathrm{ord}(h)}(r)h&=\sigma^{\mathrm{ord}(h)}(r)(X^{\mathrm{ord}(h)}q+H.T.)\\
			&=(\sigma^{\mathrm{ord}(h)}(r)X^{\mathrm{ord}(h)})q+H.T.\\
			&=X^{\mathrm{ord}(h)}\sigma^{-\mathrm{ord}(h)}(\sigma^{\mathrm{ord}(h)}(r))q+H.T.\\
			&=X^{\mathrm{ord}(h)}rq +H.T..
		\end{align*}
		where $rq=\sigma^{-\mathrm{ord}(h)}(\mathrm{lc}(\sigma^{\mathrm{ord}(h)}(r)h))\in Q$. So $Q$ is a left ideal of $R$.
		
		Now, let $f_{1}\in P$ be a nonzero element such that $f_{1}$ has the least order among elements of $P$ and define $q_{1}:=\sigma^{-\mathrm{ord}(f_{1})}(\mathrm{lc}(f_{1}))$. Next, choose $f_{2}\in P$ such that $f_{2}$ has the least order among all the elements of $P$ whose $\sigma^{-\mathrm{ord}(f_{2})}(\mathrm{lc}(f_{2}))$ is not in $\langle q_{1}\rangle_{l}$ and set $q_{2}:=\sigma^{-\mathrm{ord}(f_{2})}(\mathrm{lc}(f_{2}))$. Continue the choice inductively, i.e., choose $f_{i}\in P$ such that $f_{i}$ has the least order among all the elements of $P$ whose $\sigma^{-\mathrm{ord}(f_{i})}(\mathrm{lc}(f_{i}))$ is not in 
        $\langle q_{1},\ldots,q_{i-1} \rangle_{l}$ and set $q_{i}:=\sigma^{-\mathrm{ord}(f_{i})}(\mathrm{lc}(f_{i}))$. We can choose such $f_{i}\in P$ since the set consisting of the orders of the elements of $P$ is a subset of $\mathbb{N}$, so it satisfies the Well-Ordering Principle (WOP). This process creates an ascending chain of left ideals,
		\begin{equation*}
			I_{1}=\langle q_{1} \rangle_{l}\subset I_{2}=\langle q_{1},q_{2} \rangle_{l}\subset\cdots\subset I_{i}=\langle q_{1},q_{2},\ldots,q_{i} \rangle_{l}\subset\cdots
		\end{equation*}
		Note that $Q$ consists of all images of leading coefficients of elements of $P$ under the map $\sigma^{\mathrm{-ord}(f)}$ so it is a part of the chain. Since $R$ is left Noetherian, it satisfies ACC on left ideals, and the chain stabilizes. So, the choice process on $f_{i}$ must stop and so $Q=I_{n}$ for some $n\in\mathbb{N}$. Set $A:=\{f_{1},\ldots,f_{n}\}$ and $d_{i}:=\mathrm{ord}(f_{i})$.\\
    
		 \textbf{Claim}: $A$ generates $P$\\
         \vspace{0.2cm}
         Let $p\in P$ where $\mathrm{ord}(p)=u$ and $\mathrm{lc}(p)=v$.  Define $p_{0}:=p$. By definition of $Q$, $w := \sigma^{-u}(v) \in Q$ and since $R$ is left quasi-associative, there exists $\gamma'_{i,0} \in R$ such that 
         
         \[
         w = \sum_{i=1}^{n}\gamma'_{i,0}q_{i}
         \]
         Note that if $u < d_a$ for some minimum $a$, then $\gamma'_{i,0} = 0$ for all $i \ge a$ by the choice of $A$.
         
         Consider the element
         \[
         s_0 := \sum_{i=1}^{n} X^{u - d_i} \sigma^{d_i}(\gamma'_{i,0}) f_i\in P.
         \]
         where if we write $f_{i}$ using right-hand coefficients, $f_{i}=X^{d_i} q_i + H.T.$. Expanding $s_{0}$, we obtain
         \begin{align*}
             s_0 = \sum_{i=1}^{n} X^{u - d_i} \sigma^{d_i}(\gamma'_{i,0}) (X^{d_i} q_i + H.T.)
             &= \sum_{i=1}^{n} X^{u - d_i} \cdot X^{d_i} \gamma'_{i,0} q_i + H.T.\\
             &= X^{u} \sum_{i=1}^{n} \gamma'_{i,0} q_i + H.T.\\
             &= X^{u} w + H.T.
         \end{align*}
         Observe that $X^{u} w+H.T. = X^{u} \sigma^{-u}(v)+H.T. = v X^{u} + H.T.$. This implies that $\mathrm{ord}(s_0)=u=\mathrm{ord}(p_{0})$ and $\mathrm{lc}(s_0)=v=\mathrm{lc}(p_0)$. Define $p_1 := p_0 - s_0 \in P$. Since $\mathrm{lc}(p_0) = \mathrm{lc}(s_0)$, $p_1$ is either $0$ or $\mathrm{ord}(p_0) < \mathrm{ord}(p_1)$. Repeating this process inductively on $p_{m}\in P$ yields a sequence of elements in $P$ of the form $s_m = \sum\limits_{i=1}^{n} X^{\mathrm{ord}(p_m)-d_i} \sigma^{d_i}(\gamma'_{i,m}) f_i$ where $p_{m}:=p_{m-1}-s_{m-1}$ is either $0$ or $\mathrm{ord}(p_{m-1})<\mathrm{ord}(p_{m})$. Observe that $p=(p_{0}-p_{1})+(p_{1}-p_{2})+(p_{2}-p_{3})+\cdots=\sum\limits_{m=0}^{\infty}s_{m}$ so
         \begin{align*}
p=\sum\limits_{m=0}^{\infty}s_{m}&=\sum\limits_{m=0}^{\infty}\left(\sum\limits_{i=1}^{n} X^{\mathrm{ord}(p_m)-d_i} \sigma^{d_i}(\gamma'_{i,m}) f_i\right)\\
&=\sum\limits_{i=1}^{n}\left(\sum\limits_{m=0}^{\infty}X^{\mathrm{ord}(p_{m})-d_{i}}\sigma^{d_{i}}(\gamma_{i,m}')f_{i}\right)\\
&=\sum\limits_{i=1}^{n}\left(\sum\limits_{m=0}^{\infty}X^{\mathrm{ord}(p_{m})-d_{i}}\sigma^{d_{i}}(\gamma_{i,m}')\right)f_{i}.
         \end{align*}
			This means that $p$ can be written as a finite linear combination of elements of $A$. Hence, $P$ is a finitely generated left ideal of $R[[X;\sigma]]$ and so $R[[X;\sigma]]$ is left Noetherian.
	\end{proof}

    Theorem \ref{3110} recovers the Hilbert Basis Theorem for left Noetherianity of skew power series rings \cite{Goodearl} whenever $R$ is associative, left Noetherian, and $\sigma$ is an automorphism on $R$. Moreover, if $R$ is associative, Theorem \ref{3110} provides the version of the theorem for left Noetherianity of non-associative skew power series rings over an associative coefficient ring.

    \begin{rem}
     If $R$ is a left Noethrian non-associative ring that is not left quasi-associative and $P$ is a left ideal of $R[[X;\sigma]]$, then we can still find $f_1, \ldots f_n$ like in the proof of Theorem \ref{3110} but we can only say that if $p \in P$, then
    \[
		p=\sum\limits_{i=1}^{n}\sum\limits_{j=1}^{n_{i}}\sum\limits_{l=0}^{k_{ij}}\Bigg(\sum\limits_{m=0}^{\infty}X^{m}(t_{ijlj_{m}}'(\cdots(t_{ijl2_{m}}'(t_{ijl1_{m}}'f_{i}))\cdots))\Bigg)
		\]
		for some $n_{i}, k_{ij}\in\mathbb{N}$, $t'_{ijl1_{m}},t'_{ijl2_{m}},\ldots,t'_{ijlj_{m}}\in R$.
    \end{rem}

    This does not show that $P$ is a finitely generated left ideal of $R[[X; \sigma]]$ since
an infinite sum of elements in a left ideal need not lie in the left ideal even when
the sum is defined in $R[[X;\sigma]]$; in this case, the sum
\[
\sum\limits_{m=0}^{\infty}X^{m}(t_{ijlj_{m}}'(\cdots(t_{ijl2_{m}}'(t_{ijl1_{m}}'f_{i}))\cdots)).
\]
Factoring $f_i$ to the right out of the infinite sum is not guaranteed since $f_i$ is not necessarily in $N_r(R[[X; \sigma]])$.\\

	We have an example satisfying Theorem \ref{3110}.
    \vspace{0.1cm}

    	\begin{ex}
		\textnormal{Let $u=\sum\limits_{i=0}^{7}d_{i}e_{i}\in\mathbb{O}$. Define $\sigma:\mathbb{O}\longrightarrow\mathbb{O}$ by }
		\[
		\sigma(u):=\overline{u}
		\]
		\textnormal{where $\overline{u}$ is called the conjugate of $u$ where $\overline{u}=d_{0}e_{0}+\left(-\left( \sum\limits_{i=1}^{7}d_{i}e_{i} \right)\right)$. Obviously, $\sigma(1_{\mathbb{R}})=1_{\mathbb{R}}$ and $\sigma$ is an additive bijection map. Also, $\mathbb{O}$ is left Noetherian since its only left ideals are $\{0\}$ and $\mathbb{O}$ itself, which satisfies ACC on left ideals. From Example \ref{3114}, $\mathbb{O}$ is left quasi-associative.
		By Theorem \ref{3110}, the non-associative skew power series ring $\mathbb{O}[[X;\sigma]]$ is left Noetherian.}
	\end{ex}
\vspace{0.1cm}

	The following result describes when a non-associative skew power series ring is Noetherian.
\vspace{0.1cm}
	\begin{cor}
	Let $R$ be a non-associative ring and $\sigma: R \longrightarrow R$ be an additive map such that $\sigma(1) = 1$. If $R$ is Noetherian that is quasi-associative, and $\sigma$ is a bijective map, then $R[[X; \sigma]]$ is Noetherian.
	\end{cor}
	\begin{proof}
		The proof follows from Theorem \ref{314} and Theorem \ref {3110} using the fact that every bijective map is surjective.
	\end{proof}

    \begin{ex}
		\textnormal{Recall that $\mathbb{O}$ is Noetherian by Example \ref{211} and both left and right quasi-associative by Example \ref{3114}. Let $\sigma : \mathbb{O} \longrightarrow \mathbb{O}$ be a conjugate map which is a bijective map. Then $\mathbb{O}[[X; \sigma]]$ is Noetherian.}
	\end{ex}
\vspace{0.1cm}

The Hilbert Basis Theorem for right (resp. left) Noetherianity of $R[[X;\sigma]]$ fails in general if $R$ is not right (resp. left) quasi-associative, despite $R$ being right (resp. left) Noetherian and $\sigma$ being an additive bijective map on $R$.

\begin{ex}\label{counter}
    \textnormal{Let $K$ be a field. Define $A:=K[u]$ and $M:=K[t]$ be polynomial rings over $K$. For any $a=\sum\limits_{i=0}^{w}a_{i}u^{i}\in A$, define a $K$-linear map $\rho_{a}:M\longrightarrow M$ by $\rho_{a}(m)=a_{0}m+a_{1}tm$ for every $m\in M$. Define $R$ to be the direct product $A\times M$ with usual componentwise addition and multiplication is defined as: for any $b=\sum\limits_{i=0}^{v}b_{i}u^{i}\in A$,} 
    \[
(a,m)(b,n)
=(ab,a_{0}n+\rho_{b}(m))=
\bigl(ab,\;a_0n+b_0m+b_1tm\bigr).
\]
\textnormal{$R$ is an additive abelian group and the multiplication is distributive over addition so $R$ is a non-associative ring. Particularly, the multiplication here is not associative since $((0,1)(u,0))(u,0)=(0,t^{2})\neq(0,0)= (0,1)((u,0)(u,0))$ and obviously not commutative. The identity element here is $(1,0)$.} 

\textnormal{$R$ is not right quasi-associative by considering the elements 
$r=(0,1)$,
$b=c=(u,0)\in R$.}
%\[
%(rb)c=(0,t)(u,0)=(0,\rho_{u}(t))=(0,t^{2}).
%\]
%\textnormal{Suppose $R$ is right quasi-associative. Let \(s=(a,n)\in R\) such that $(rb)c=rs$.}
%\textnormal{Then}
%\[
%rs
%=
%(0,1)(a,n)
%=
%\bigl(0,\rho_{a}(1)\bigr)
%=
%(0,a_0+a_1t).
%\]

%\textnormal{Consequently,
%$rR
%=
%\{(0,c_{1}+c_{2}t):c_{1},c_{2}\in K\}$ %where $(0,t^2)\notin rR$. Therefore, %there is no \(s\in R\) such that %$(rb)c=rs$.}

\textnormal{Define $M_{0}:=0\times M
=
\{(0,m):m\in M\}$ which is a right ideal of $R$. The quotient $R/M_{0}\cong A$. Since $A$ is right Noetherian, then so is $R/M_{0}$. We show that $M_{0}$ is also right Noetherian. Let $J\subseteq M_{0}$ be a right ideal of $R$. We can identify $M_{0}$ with $M$ via $ (0, m) \longleftrightarrow m$.}
Define \[
\tilde{J} = \{m \in M : (0,m) \in J\}.
\]
\textnormal{We claim that $\tilde{J}$ is a right ideal of $M$. To see this, $\tilde{J}$ is an additive subgroup of $M$ since for every $a,b\in\tilde{J}$, $(0,a),(0,b)\in J$. So, $(0,a)-(0,b)=(0,a-b)\in J$ so $a-b\in\tilde{J}$. Next, let $(u,0)\in R$ and $(0,m)\in J$. Since $J$ is a right ideal of $R$,}
\begin{align*}
    (0,m)(u,0)=(0,0+\rho_{u}(m))=(0,0+0m+1tm)=(0,tm)\in J.
\end{align*}
\textnormal{This means $tm\in\tilde{J}$. Inductively, $t^{i}m\in\tilde{J}$ for all $i\geq 0$. Let $f=\sum\limits_{i=0}^{w}c_{i}t^{i}\in M$. Note that $m,tm,\ldots,t^{w}m\in\tilde{J}$. Also, for any scalar $c\in K$, $(0,m)(c,0)=(0,0+\rho_{c}(m))=(0,0+cm+0tm)=(0,cm)\in J$. So, $cm\in\tilde{J}$. Since we already established that $\tilde{J}$ is an additive subgroup of $M$,} 
\[
fm=\left(\sum\limits_{i=0}^{w}c_{i}t^{i}\right)m=\sum\limits_{i=0}^{w}c_{i}t^{i}m\in\tilde{J}.
\]
\textnormal{Since $M$ is commutative, $fm=mf\in\tilde{J}$. By arbitrariness of $m\in\tilde{J}$ and $f\in M$, $\tilde{J}$ is a right ideal of $M$.}\\

\textnormal{Since $M$ is a principal ideal domain, $\tilde{J}$ is a principal right ideal. Now, let
\[
J_1 \subset J_2 \subset J_3 \subset \cdots
\]
be an ascending chain of right ideals of $R$ contained in $M_0$. From the identification above, the
corresponding sets
\[
\tilde{J}_1 \subseteq \tilde{J}_2 \subseteq \tilde{J}_3 \subseteq \cdots
\]
also form an ascending chain of right ideals of $M$. Since $M$ is right Noetherian, there exists $m\in\mathbb{N}$ such that $\tilde{J}_n = \tilde{J}_m$ for all $n \ge m$. Therefore $J_n = J_m $ for all $n \ge m$. Hence, the right ideals of $R$ contained in $M_0$ satisfy ACC. So, $M_{0}$ is right Noetherian.}

\textnormal{We now show that $R$ is right Noetherian. Suppose
\[
J_1\subset J_2\subset J_3\subset\cdots
\]
is an ascending chain of right ideals of \(R\). Consider the respective ascending chain of right ideals
\[
J_1\cap M_0
\subset
J_2\cap M_0
\subset
J_3\cap M_0
\subset\cdots
\]
in $M_{0}$ and
\[
(J_1+M_0)/M_0
\subset
(J_2+M_0)/M_0
\subset
(J_3+M_0)/M_0
\subset\cdots
\]
in $R/M_{0}$. The first chain stabilizes because \(M_0\) is right Noetherian. The second chain also stabilizes since $A$ is right Noetherian. Hence, there is \(m\in\mathbb{N}\) such that, for all \(n\geq m\), $J_n\cap M_0=J_m\cap M_0$
and $(J_n+M_0)/M_0=(J_m+M_0)/M_0$. We claim that $J_n=J_m$ for all $n\geq N$. Let \(x\in J_n\). Since \(R/M_0\) is right Noetherian, there exists \(y\in J_m\) such that $x+M_0=y+M_0$. Hence, $x-y\in M_0$. Since \(x\in J_n\) and \(y\in J_m\subseteq J_n\), we also have $x-y\in J_n$. Therefore, $x-y\in J_n\cap M_0
=
J_m\cap M_0
\subseteq J_m$. Since \(y\in J_m\), so $x=(x-y)+y\in J_m$. Thus, $J_n\subseteq J_m$. The reverse inclusion follows from the original ascending chain,
so $J_n=J_m$. Consequently, $
R$ is right Noetherian.}

\textnormal{Consider the non-associative skew power series ring $S:=R[[X;\sigma]]$ where $\sigma$ is the identity map on $R$. We claim here that $S$ is not right Noetherian. For every $k\geq 1$, define $p_k=\sum\limits_{j=0}^{\infty}
(0,t^{kj})X^j$. For every $n\geq 1$, define $J_{n}=\langle p_1,\ldots,p_n\rangle_{r}$ which is a right ideal of $S$. Observe that 
\[
J_{1}\subset J_{2}\subset J_3\subset\cdots.
\]
is an ascending chain of right ideals in $S$. We claim that for every \(n\geq2\), $p_n\notin J_{n-1}$. Suppose on the contrary that $p_n\in J_{n-1}$. Then $p_{n}$ is a finite sum consisting of terms which are products of the form \[(\cdots((p_kq_1)q_2)\cdots) q_d,\qquad (3.5)\]
where $1\leq k\leq n-1$ and $q_1,\ldots,q_d\in S$ for some $d\in\mathbb{N}$. Since $R$ is not right quasi-associative, by the contrapositive of Proposition \ref{inherit}, $S$ is not right quasi-associative. This implies that there exists $D\in\mathbb{N}$ such that the finite sum expressing $p_{n}$ uses at most $D\geq 2$ right multiplications of $q_{i}$ where $d\leq D$. Observe that the coefficient of $X^{j}$ in $p_{k}$ is $(0,t^{kj})$ so the degree of $t^{kj}$, $\deg(t^{kj})=kj$. Let $N\in\mathbb{N}$. If $0\leq j\leq N$, then $kj\leq (n-1)N$. This implies that $\max\limits_{0\leq j\leq N}kj
\leq
(n-1)N$. Let $(0,h)\in R$ be the coefficient of $X^{j}$ in (3.5). Then $\deg(h)\leq (n-1)N+D$. Since adding each term of the form given by (3.5) does not increase the degree of polynomials in the second component of the coefficient, this means that in the finite sum expressing $p_{n}$, the $X^{N}$-coefficient of $p_{n}$, which is $(0,t^{nN})$ also satisfies $\deg(t^{nN})=nN\leq (n-1)N+D$. This implies that $N\leq D$. Now, if we choose $N>D$, then this case implies that $\deg(t^{nN})=nN> (n-1)N+D$. This is a contradiction to the established upper bound $(n-1)N+D$. Hence,
$p_n\notin J_{n-1}$ and so $J_{n-1}\subset J_{n}$ for every $n\geq 2$, that is, this ascending chain would not stabilze. So, $S$ does not satisfy the ACC on right ideals. Hence, $S$ is not right Noetherian.}
\end{ex}

\begin{ex}
\textnormal{For the left version, we still consider the same set $R:=A\times M$, but the multiplication is now defined as}
\[
(a,m)(b,n)=(ab,b_{0}m+\rho_{a}(n))=(ab,b_{0}m+a_{0}n+a_{1}tn)
\]
\textnormal{By applying the analogous reason from Example \ref{counter}, we can say that $S:=R[[X;\sigma]]$ is also not left Noetherian.}
\end{ex}

Now, let us explore the Hilbert Basis Theorem for  $R((X;\sigma))$. The following proposition exhibits how left and right ideals of $R((X;\sigma))$ relate to $R[[X;\sigma]]$.
\vspace{0.1cm}

\begin{prop}\label{3117}
Let $R((X;\sigma))$ be a non-associative skew Laurent series ring. If $P$ is a right ideal of $R((X;\sigma))$, then $P=(P\cap R[[X;\sigma]])R((X;\sigma))$. If $P$ is a left ideal of $R((X;\sigma))$, then $P=R((X;\sigma))(P\cap R[[X;\sigma]])$.

\end{prop}

\begin{proof}
	Let $R$ be a non-associative ring and let $\sigma: R\longrightarrow R$ be an additive bijection such that $\sigma(1_{R})=1_{R}$. Let $P$ be a right ideal of $R((X;\sigma))$. Let $z\in (P\cap R[[X;\sigma]])R((X;\sigma))$. Then $z=ab$ for some $a\in P\cap R[[X;\sigma]]$, $b\in R((X;\sigma))$. Observe that $(P\cap R[[X;\sigma]])R((X;\sigma))\subseteq PR((X;\sigma))\subseteq P$. Thus, $z\in P$ and so $	(P\cap R[[X;\sigma]])R((X;\sigma))\subseteq P$. Let $z\in P$. Then \[z=\sum\limits_{i=m}^{\infty}z_{i}X^{i}=z_{m}X^{m}+H.T.\] for some $m\in\mathbb{Z}$. Now, consider the element $zX^{-m}\in P$. Note that
	\[
	zX^{-m}=z_{m}+H.T
	\]
	where $\mathrm{ord}(zX^{-m})=0$ and the powers of $X$ in H.T. of $zX^{-m}$ are positive integers since for every $X^{k}$ in H.T. of $z$, we have $X^{k}X^{-m}=X^{k-m}$ where $k-m>0$ due to $\mathrm{ord}(z)=m$. So, $zX^{-m}\in R[[X;\sigma]]$. This implies that $zX^{-m}\in P\cap R[[X;\sigma]]$. Now, $z=(zX^{-m})X^{m}$ where $X^{m}\in R((X;\sigma))$. So, $z\in (P\cap R[[X;\sigma]])R((X;\sigma))$ and so $P\subseteq  (P\cap R[[X;\sigma]])R((X;\sigma))$. Hence, the equality holds. The analogous argument follows for left ideals.
\end{proof}

Now, we have the following theorem that shows the interdependency between the right (resp. left) Noetherianity of $R((X;\sigma))$ and $R[[X;\sigma]]$, provided that a non-associative ring $R$ is right (resp. left) quasi-associative.

\begin{thm}\label{3118}
Let $R$ be a non-associative ring that is right (resp. left) quasi-associative and
$\sigma:R\longrightarrow R$ be an additive bijection map such that $\sigma(1)=1$. A non-associative skew Laurent series ring $R((X;\sigma))$ is right (resp. left) Noetherian if and only if the non-associative skew power series ring $R[[X;\sigma]]$ is right (resp. left) Noetherian.
\end{thm}
\begin{proof}
Assume $R((X;\sigma))$ to be right Noetherian. Let 
\[
I_{1}\subset I_{2}\subset\cdots\qquad\text{(3.7)}
\]
be an ascending chain of right ideals in $R$. For each $i\in\mathbb{N}$, define $I_{i}'$ as the right ideal of $R((X;\sigma))$ generated by $I_{i}$. This creates an ascending chain of right ideals in $R((X;\sigma))$, i.e., $I_{1}'\subset I_{2}'\subset\cdots$. Here, $I_{i}\subseteq I_{i}'$. Since $R((X;\sigma))$ is right Noetherian, there exists $m\in\mathbb{N}$ such that $I_{m}'=I'_{m+1}$. From (3.7), $I_{m}\subseteq  I_{m+1}$. Now, let $c\in I_{m+1}$. Since $I_{m+1}\subseteq I_{m+1}'$, $c\in I_{m+1}'$. So, $c\in I_{m}'$. Thus, $c$ can be expressed as
\[
c=\sum_{i=1}^{n} \sum_{j=1}^{n_i} \sum_{l=1}^{k_{ij}} (\cdots((a_i r_{ijl1})r_{ijl2})\cdots)r_{ijlj}\qquad\text{(3.8)}
\]
for some $n,n_{i},k_{ij}\in\mathbb{N}$, $r_{ijl1}=s_{ijl1}X^{d_{1}}+H.T., r_{ijl2}=s_{ijl2}X^{d^{2}}+H.T.,\cdots, r_{ijlj}=s_{ijlj}X^{d_{j}}+H.T.\in R((X;\sigma))$, $a_{i}\in I_{m}$. Since $c$ is a constant, it has order $0$ so the right-hand side of (3.8) is also of order $0$. Then
\begin{align*}
	c&=\sum_{i=1}^{n} \sum_{j=1}^{n_i} \sum_{l=1}^{k_{ij}} (\cdots((a_i r_{ijl1})r_{ijl2})\cdots)r_{ijlj}\\
	&=\sum_{i=1}^{n} \sum_{j=1}^{n_i} \sum_{l=1}^{k_{ij}}((\cdots((a_{i}(s_{ijl1}X^{d_{1}}))s_{ijl2}X^{d_{2}})\cdots)s_{ijlj}X^{d_{j}})+H.T.\\
	&=\sum_{i=1}^{n} \sum_{j=1}^{n_i} \sum_{l=1}^{k_{ij}}((\cdots((a_{i}s_{ijl1})\sigma^{d_{1}}(s_{ijl2}))\cdots)\sigma^{d_{j-1}}(s_{ijlj})X^{d_{1}+d_{2}+\cdots+d_{j}})+H.T.\\
	&=\sum_{i=1}^{n} \sum_{j=1}^{n_i} \sum_{l=1}^{k_{ij}}((\cdots((a_{i}s_{ijl1})\sigma^{d_{1}}(s_{ijl2}))\cdots)\sigma^{d_{j-1}}(s_{ijlj}).
\end{align*}
Since $a_{i}\in I_{m}$ and $I_{m}$ is closed under addition and satisfies right absorption law, $c\in I_{m}$. Therefore, $I_{m+1}\subseteq I_{m}$ and so $I_{m}=I_{m+1}$. This implies that $R$ is right Noetherian. Now, $R$ is right quasi-associative by assumption and $\sigma:R\longrightarrow R$ is obviously a surjection map by definition of $R((X;\sigma))$, so by Hilbert Basis Theorem for right Noetherianity of $R[[X;\sigma]]$, $R[[X;\sigma]]$ is right Noetherian.

			Conversely, suppose a non-associative skew power series ring $R[[X;\sigma]]$ is right Noetherian. Then every right ideal of $R[[X;\sigma]]$ is finitely generated. Let $P$ be a right ideal of $R((X;\sigma))$ and $h\in R((X;\sigma))$. Note that $P\cap R[[X;\sigma]]$ is a right ideal of $R[[X;\sigma]]$. By assumption on $R[[X;\sigma]]$, $P\cap R[[X;\sigma]]$ is finitely generated as a right ideal of $R[[X;\sigma]]$, say by $f_{1},\ldots, f_{n}\in P\cap R[[X;\sigma]]$. Let $a\in P\cap R[[X;\sigma]]$. Then
\begin{align*}
ah&=\left(\sum\limits_{i=1}^{n}\sum\limits_{j=1}^{n_{i}}\sum\limits_{l=1}^{k_{ij}}(\cdots((f_{i}w_{ijl1})w_{ijl2})\cdots)w_{ijlj}\right)h\\
&=\sum\limits_{i=1}^{n}\sum\limits_{j=1}^{n_{i}}\sum\limits_{l=1}^{k_{ij}}((\cdots((f_{i}w_{ijl1})w_{ijl2})\cdots)w_{ijlj})h\qquad \qquad\text{(3.9)}
\end{align*}
for some $w_{ijl1},w_{ijl2},\ldots,w_{ijlj}\in R[[X;\sigma]]$. So, $ah\in(P\cap R[[X;\sigma]])R((X;\sigma))$. By Proposition \ref{3117}, $P=(P\cap R[[X;\sigma]])R((X;\sigma))$ so $ah\in P$. Since $R[[X;\sigma]]\subseteq R((X;\sigma))$, $w_{ijl1},w_{ijl2},\ldots,w_{ijlj}\in R((X;\sigma))$. By arbitrariness of $ah\in P$ and by (3.9), $P$ is a finitely generated right ideal of $R((X;\sigma))$. Hence, $R((X;\sigma))$ is right Noetherian.

The left version follows analogously.
\end{proof}

 We now present the Hilbert Basis Theorem for right Noetherianity of non-associative skew Laurent series ring $R((X;\sigma))$.

\begin{cor} \label{3119}
	Let $R$ be a non-associative ring and $\sigma:R\longrightarrow R$ be an additive bijection map such that $\sigma(1)=1$. If $R$ is a right Noetherian ring that is right quasi-associative, then $R((X; \sigma))$ is right Noetherian.
\end{cor}

\begin{proof}
	By Theorem \ref{314} and Theorem \ref{3118}, the conclusion follows.
\end{proof}

\begin{ex}\label{last1}
   \textnormal{Consider the non-associative skew Laurent series ring $\mathbb{O}((X;\sigma))$, where $\mathbb{O}$ is a non-associative right Noetherian ring that is right quasi-associative by Example \ref{211} and Example \ref{3114}, respectively, and $\sigma$ is defined to be the conjugate map on $\mathbb{O}$, which is a bijective map. By Corollary \ref{3119}, $\mathbb{O}((X;\sigma))$ is right Noetherian.}
\end{ex}

Analogously, we have the Hilbert Basis Theorem for left Noetherianity

\begin{cor}\label{3120left}
 Let $R$ be a non-associative ring and $\sigma:R\longrightarrow R$ be an additive bijection map such that $\sigma(1)=1$. If $R$ is a left Noetherian ring and $R$ is left quasi-associative, then $R((X;\sigma))$ is left Noetherian.	
\end{cor}

\begin{proof}
	By Theorem \ref{3110} and Theorem \ref{3118}, the conclusion follows.
\end{proof}

\begin{ex}
 \textnormal{Define the non-associative ring $R:=\mathbb{O}\times\mathbb{H}$ where $\mathbb{H}$ is the set of quaternions over $\mathbb{R}$ and $\sigma:R\longrightarrow R$ to be 
 $\sigma((r,s)):=(\overline{r},\overline{s})$. By Noetherianity of direct product of non-associative rings \cite{back2023hilberts}, $R$ is Noetherian since $\mathbb{O}$ is Noetherian and $\mathbb{H}$ is Noetherian, implying that $R$ is left Noetherian. By Remark \ref{rem}, $R$ is left quasi-associative. On the other hand, $\sigma$ preserves the identity element, that is, $\sigma(1_{\mathbb{R}},1_{\mathbb{R}})=(\overline{1_{\mathbb{R}}},\overline{1_{\mathbb{R}}})=(1_{\mathbb{R}},1_{\mathbb{R}})$, and $\sigma$ is an additive bijection map since each component of $\sigma$ is the conjugate map on $\mathbb{O}$ and $\mathbb{H}$, respectively. Then consider the non-associative skew Laurent series ring, $R((X;\sigma))$. By Corollary \ref{3120left}, $R((X;\sigma))$ is left Noetherian.}
\end{ex}

Now, we can describe when a non-associative skew Laurent series ring $R((X;\sigma))$ is Noetherian.

\begin{cor}\label{3120}
	Let $R$ be a non-associative ring and $\sigma:R\longrightarrow R$ be an additive bijection map such that $\sigma(1)=1$. If $R$ is a Noetherian ring that is quasi-associative, then $R((X;\sigma))$ is Noetherian.
\end{cor}

\begin{proof}
     This follows from  Corollary \ref{3119} and Corollary \ref{3120left}.
\end{proof}

\begin{ex}
    \textnormal{Define the non-associative ring $R:=\mathbb{O}\times\mathbb{Z}$ and $\sigma:R\longrightarrow R$ to be 
 $\sigma((r,n)):=(\overline{r},n)$. By Noetherianity of direct product of non-associative rings \cite{back2023hilberts}, $R$ is Noetherian since $\mathbb{O}$ is Noetherian and $\mathbb{Z}$ is Noetherian. By Remark \ref{rem}, $R$ is quasi-associative. On the other hand, $\sigma$ preserves the identity element, that is, $\sigma(1_{\mathbb{R}},1_{\mathbb{Z}})=(\overline{1_{\mathbb{R}}},1_{\mathbb{Z}})=(1_{\mathbb{R}},1_{\mathbb{Z}})$ and $\sigma$ is an additive bijection map since $\sigma$ can be redefined as $\sigma((r,n)):=(\phi(r),\mathrm{id}(n))=(\overline{r},n)$ where $\phi:\mathbb{O}\longrightarrow\mathbb{O}$ is the conjugate map and $\mathrm{id}:\mathbb{Z}\longrightarrow\mathbb{Z}$ is the identity map. Componentwise, these maps are additive and bijective. Now, consider the non-associative skew Laurent series ring, $R((X;\sigma))$. By Corollary \ref{3120}, $R((X;\sigma))$ is Noetherian.} 
\end{ex}

Considering $S:=R((X;\sigma))$ in Example \ref{counter} (resp. Example 4.12) and following the same argument, the Hilbert Basis Theorem for right (resp. left) Noetherianity of $R((X;\sigma))$ also fails in general if $R$ is not right (resp. left) quasi-associative, despite $R$ being right (resp. left) Noetherian and $\sigma$ being an additive bijective map on $R$.\\

\noindent {\bf Acknowledgment.} The authors express their utmost gratitude and appreciation to the distinguished organization of the Department of Science and Technology-Accelerated Science and Technology Human Resource Development Program (DOST-ASTHRDP) in the Philippines for providing the necessary financial support needed in the conduct of this study.\\

\noindent {\bf AI use.} The authors have consulted AI tools for help with editing and literature search. \\

\noindent {\bf Disclosure Statement.} The authors declare that they have no competing interests.

\end{document}